\documentclass[11pt,reqno]{amsart}

\usepackage{amssymb,amsmath}
\usepackage{tikz}
\usepackage{hyperref}
\usepackage{url}
\usepackage{color}
\usepackage{breakurl}
\newcommand{\bburl}[1]{\textcolor{blue}{\url{#1}}}
\usepackage[letterpaper,margin=1in]{geometry}

\def\QQ{\mathbb Q}
\def\F{\mathbb F}
\def\E{\mathcal{E}}
\def\OK{\mathcal{O}_K}
\newcommand{\p}{\mathfrak p}

\newcommand{\legsym}[2]{\left(\frac{#1}{#2}\right)}

\theoremstyle{plain}
\newtheorem{theorem}{Theorem}[section]
\newtheorem{cor}[theorem]{Corollary}
\newtheorem{prop}[theorem]{Proposition}
\theoremstyle{definition}
\newtheorem{remark}[theorem]{Remark}
\newtheorem{example}[theorem]{Example}
\numberwithin{equation}{section}
\numberwithin{theorem}{section}

\begin{document}

\title{Constructing families of moderate-rank elliptic curves over number fields}

\author[Mehrle]{David Mehrle}
  \email{\textcolor{blue}{\href{mailto:dmehrle@math.cornell.edu}{dmehrle@math.cornell.edu}}}

\author[Miller]{Steven J. Miller}
\email{\textcolor{blue}{\href{mailto:sjm1@williams.edu}{sjm1@williams.edu}},  \textcolor{blue}{\href{Steven.Miller.MC.96@aya.yale.edu}{Steven.Miller.MC.96@aya.yale.edu}}}
\address{Department of Mathematics and Statistics, Williams College, Williamstown, MA 01267}

\author[Reiter]{Tomer Reiter}
\email{\textcolor{blue}{\href{mailto:treiter@andrew.cmu.edu}{treiter@andrew.cmu.edu}}}

\author[Stahl]{Joseph Stahl}
\email{\textcolor{blue}{\href{mailto:jstahl@bu.edu}{jstahl@bu.edu}}}

\author[Yott]{Dylan Yott}
\email{\textcolor{blue}{\href{mailto:dyott@math.berkeley.edu}{dyott@math.berkeley.edu}}}



\date{\today}

\thanks{This work was supported by NSF grants DMS-1347804, DMS-1265673, Williams College, and the PROMYS program. The authors thank \'Alvaro Lozano-Robledo, Rob Pollack and Glenn Stevens for their insightful comments and support. Thanks also to the referee for their careful reading of an earlier version of this paper.}


\begin{abstract} 
We generalize a construction of families of moderate rank elliptic curves over $\QQ$ to number fields $K/\QQ$. The construction, originally due to Scott Arms, \'Alvaro Lozano-Robledo and Steven J. Miller, invokes a theorem of Rosen and Silverman to show that computing the rank of these curves can be done by controlling the average of the traces of Frobenius; the construction for number fields proceeds in essentially the same way. One novelty of this method is that we can construct families of moderate rank without having to explicitly determine points and calculating determinants of height matrices.
\end{abstract}

\maketitle

\section{Introduction}

If $E$ is an elliptic curve over $\QQ$, then the associated group of rational solutions, the Mordell-Weil group $E(\QQ)$, is finitely generated. The rank of this group is a very interesting and well-studied quantity in modern number theory; the famous Birch and Swinnerton-Dyer conjecture states that its rank equals the order of vanishing of the elliptic curve's $L$-function at the central point. We assume the reader is familiar with the basics of the subject; good references are \cite{Kn,Si1,Si2,SiTa}.

It is unknown if the rank of an elliptic curve over $\QQ$ can be arbitrarily large. It is an interesting and difficult problem to find examples or families of curves with large rank. To date the best known results are due to Elkies, who constructed an elliptic curve of rank at least 28 (or exactly 28 subject to the Generalized Riemann Hypothesis \cite{KSW}) and a family of elliptic curves of rank at least 18; see \cite{BMSW} for a survey of recent results on the distribution of ranks of curves in families, and conjectures for their behavior. 

Many of the constructions of high rank families of elliptic curves begin by forcing points to lie in the curves, and then calculating the associated height matrices to verify that they are linearly independent (see for example \cite{Mes1,Mes2,Na1}). We pursue an alternative approach introduced by Arms, Lozano-Robledo and Miller \cite{AL-RM}. Briefly, their strategy is to use a result of Rosen and Silverman \cite{RoSi}, which converts the problem of constructing families of elliptic curves with large rank to finding associated Legendre sums that are large. While in general these sums are intractable, for some carefully constructed families these can be determined in closed form, which allows us to determine the rank of the families \emph{without} having to list points and compute height matrices. Our main result is to generalize the work in \cite{AL-RM} from elliptic curves over $\QQ$ to elliptic curves over number fields. Specifically, we show the following. 

\begin{theorem}\label{thm:main} Let $K$ be a number field. Then there exists an elliptic curve $\mathcal E$ over $K(T)$ with $j(\mathcal E) \not \in \QQ(T)$ such that the rank of $\mathcal E$ over $K(T)$ is exactly 6.
\end{theorem}

By specializing to $T = t$ for some $t \in K$, we obtain curves $\mathcal E_t$ over $K$ from the curve $\mathcal E$ over $K(T)$. Silverman's specialization theorem \cite[Theorem 11.4]{Si2} tells us that for all but finitely many $t \in T$, the rank can only possibly increase.

\begin{cor}
  Let $K$ be a number field. There are infinitely many elliptic curves over $K$ with rank at least 6. 
\end{cor}

\begin{remark}
  Arms et al.\ \cite{AL-RM} construct infinitely many elliptic curves over $\QQ$ with rank at least 6. By base-extending these curves to $K$, we may trivially obtain infinitely many elliptic curves over $K$ with rank at least $6$. Our contribution is to construct curves that are defined over $K$ but \emph{not} defined over $\QQ$; this is evident because the $j$-invariant of curves we construct in Theorem \ref{thm:main} lies in $K(T)$ but not in $\QQ(T)$.
\end{remark}

\section{The Construction}

Let $K$ be a number field and $\OK$ its ring of integers. Let $\mathcal{E}$ be the elliptic curve over $K(T)$ defined by
\[
  \label{eq:ellipticcurve}
  \mathcal{E} \colon \hspace{1em} y^2 + a_1(T)xy + a_3(T)y \ =\ x^3 + a_2(T)x^2 + a_4(T)x  + a_6(T),
\]
where $a_i(T) \in \OK(T)$. By Silverman's specialization theorem \cite[Theorem 11.4]{Si2}, for all but finitely many $t \in \OK$ the Mordell-Weil rank of the fiber $\mathcal{E}_t$ over $K$ is at least that of the rank of $\mathcal{E}$ over $\OK(T)$. Therefore, if we can compute the rank of $\mathcal{E}$, we have a family of infinitely many curves $\mathcal{E}_t$ over $K$ with at least the rank of $\mathcal{E}$.

To that end, for $\mathcal{E}$ as above and $\p$ a prime of good reduction in $\OK$ (we do not consider the bad primes here), we define the average
\begin{equation}
  \label{eq:avgnumpoints}
  A_{\mathcal{E}}(\p)\ :=\ \frac{1}{N(\p)} \sum_{t \in \OK / \p} a_{t}(\p),
\end{equation}
where $N(\p) = | \OK / \p |$ and $a_t(\p) = N(\p) + 1 - \#\mathcal{E}_t(\OK / \p)$. Nagao \cite{Na2} conjectured that these sums are related to the rank of the family of elliptic curves. Rosen and Silverman proved this conjecture when $\mathcal{E}$ is a rational elliptic surface \cite{RoSi}. Specifically, whenever Tate's conjecture holds (which is known for $K3$ surfaces over certain fields \cite{Sr}) we have
\begin{equation}
  \label{eq:rosensilvermanthm}
  \lim_{X \to \infty} \frac{1}{X} \sum_{\p \colon N(\p) \leq X} -A_\mathcal{E}(\p) \log(N(\p))\ =\ \text{rank } \mathcal{E}(K(T)).
\end{equation}

Below we study certain carefully chosen families where we are able to prove that $A_\mathcal{E}(\p) = -6$ for almost all primes $\p$, thus proving these families have rank 6. To calculate the limit \eqref{eq:rosensilvermanthm}, we appeal to the Landau Prime Ideal Theorem, a generalization of the Prime Number Theorem.

\begin{theorem}[Landau Prime Ideal Theorem \cite{Lan}] We have
  \[
    \sum_{\p \colon N(\p) \leq X} \log(N(\p)) \ \sim\ X. 
  \]
\end{theorem}

Assuming we can produce $\mathcal{E}$ such that $A_{\mathcal{E}} (\p) = -6$ for almost all $\p$, then combining the Landau Prime Ideal Theorem with equation \eqref{eq:rosensilvermanthm} it follows that
\[ 
  \text{rank } \mathcal{E}(K(T))\ =\ 6, 
\]
which completes the proof of Theorem \ref{thm:main}.

So it remains to show that we can produce an $\mathcal{E}$ such that $A_{\mathcal{E}}(\p) = -6$. As in Equation 2.2 of \cite{AL-RM}, define 
\begin{eqnarray}
  \label{eq:curvedef}
  y^2\ =\ f(x,T) &=& x^3 T^2 + 2g(x)T - h(x) \notag\\
            g(x) &=& x^3 + ax^2 + bx + c \label{eq:curveformula}\\
            h(x) &=& (A-1)x^3 + Bx^2 + Cx + D \notag\\
          D_T(x) &=& g(x)^2 + x^3h(x).\notag
\end{eqnarray}

Notice that $D_T(x)$ is one-quarter of the discriminant of $f(x,T)$, considered as a quadratic polynomial in $T$. When we specialize to a particular $t \in \OK$, we write $D_t(x)$ for the one-quarter of the discriminant of $f(x,t)$. As a degree six polynomial in $x$, write $r_1, r_2, \ldots, r_6$ for the roots of $D_T(x)$. We will see that the number of distint, nonzero roots of $D_T(x)$ control the rank of the curve. 

In order to show our claim for the elliptic curve $y^2 = f(x,T)$, we must pick six distinct, nonzero roots of $D_T(x)$ which are squares in $\OK$.  We also need the analogue of equation 2.1 from \cite{AL-RM} for number fields, which can be stated as follows. For $a$, $b$ both not zero mod $\p$ and $N(\p) > 2$, then for $t \in \OK$
\begin{equation}
  \label{eq:legendresum}
  \sum_{t \in \OK / \p} \left( \frac{at^2 + bt + c}{\p} \right)\ =\
  \begin{cases}
    (N(\p) - 1) \left( \frac{a}{\p} \right) & \text{if } (b^2 - 4ac) \in \p \\
    - \left( \frac{a}{\p} \right) & \text{otherwise.}
  \end{cases} 
\end{equation}

Note that \eqref{eq:legendresum} is already demonstrated when $\OK/\p = \F_p$ is a finite field of prime order $p$ in Lemma A.2 of  \cite{AL-RM} (they give two proofs; the result also appears in \cite{BEW}). Therefore, it suffices to show \eqref{eq:legendresum} when $\OK/\p = \F_q$ is a finite field of order $q = p^r$ for $p$ prime and $r > 1$; we do so in Proposition \ref{importantlemma} in the next section. 

Write $\F_q = \OK/\p$ for the residue field of $\OK$ at $\p$.
 We have for the fiber of the elliptic surface $y^2 = f(x,T)$ at $T = t$
\[
  a_t(\p) = - \sum_{x \in \OK/\p} \left(\frac{f(x,t)}{\p}\right) 
          = - \sum_{x \in \OK/\p} \left(\frac{x^3 t^2 + 2g(x)t - h(x)}{\p}\right), 
\]
where $\left(\frac \cdot \p\right)$ is the Legendre symbol of the residue field $\OK/\p$. 

Now we study $- N(\p) A_\E(p) = \sum_{t \in \OK/\p} a_t(\p) = \sum_{x, t \in \OK/\p} \left(\frac{f(x,t)}{\p}\right)$ to calculate $A_\E(\p)$, as in equation \eqref{eq:avgnumpoints}. When $x \in \p$, the $t$-sum vanishes unless $c \in \p$ --- it is just $\sum_{t \in \OK/\p} \left(\frac{2ct - D}{\p}\right)$. 

Assume now $x \not \in \p$. Then by \eqref{eq:legendresum}, we have 
\[
  \sum_{t \in \OK/\p} \left(\frac{x^3 t^2 + 2g(x) t - h(x)}{\p}\right) = 
    \begin{cases}
      (N(\p) - 1) \left(\frac{x^3}{\p}\right) & \text{if } D_t(x) \in \p \\
      - \left( \frac{x^3}{\p} \right)         & \text{otherwise.} 
    \end{cases}
\]
If the roots $r_1, r_2, \ldots, r_6$ are squares in $\OK$, then their contribution to the rank is $(N(\p) - 1)\left(\frac{r_i^3}{\p}\right)$. If the $r_i$ are not squares, then $\left(\frac{r_i}{\p}\right)$ will be $1$ for half of the primes of $\OK$ and $-1$ for the other half, and therefore yield no net contribution to the rank. 

So assume that we may choose coefficients $a,b,c,A,B,C,D$ such that $D_T(x)$ has six distinct, non-zero roots $r_i \in \OK$, each of which is a square. Write $r_i = \rho_i^2$ for $i = 1, \ldots, 6$. Then 
\begin{align*}
  - N(\p) A_\E(\p) & = \sum_{\substack{x \in \OK/\p \\ t \in \OK/\p}} \left(\frac{f(x,t)}{\p}\right)              =  \sum_{\substack{x \in \OK/\p \\ t \in \OK/\p}} \left(\frac{x^3t^2 + 2g(x) t - h(x)}{\p}\right) \notag \\
  & = \sum_{\substack{x \colon D_t(x) \in \p\\ t \in \OK/\p}} \left(\frac{f(x,t)}{\p}\right) + \sum_{\substack{x \colon D_t(x) \not \in \p \\ t \in \OK/\p}} \left(\frac{f(x,t)}{\p}\right) \\
  & =  6(N(\p)-1) - \sum_{x \colon D_t(x) \not \in \p} \left(\frac{x^3}{\p}\right) \notag\\
  & =  6(N(\p)-1) + 6 = 6N(\p)\notag
\end{align*}
Hence, $- N(\p)A_{\mathcal{E}}(\p) = 6 N(\p)$. Therefore $A_{\mathcal{E}}(\p) = -6$, completing the proof of Theorem \ref{thm:main}.

Now we must find $a, b, c, A, B, C, D \in \OK$ such that $D_T(x)$ has six distinct, nonzero roots $r_i = \rho_i^2$: 
\begin{align}
  D_T(x) &= g(x)^2 + x^3 h(x) \notag \\
  & = Ax^6 + (B + 2a)x^5 + (C + a^2 + 2b) x^4 + (D + 2ab + 2c) x^3 \notag\\
  & \hspace*{1cm} + (2ac + b^2) x^2 + (2 bc) x + c^2 \label{eq:discriminantcoefficients} \\
  & = A(x^6 + R_5x^5 + R_4x^4 + R_3x^3 + R_2 x^2 + R_1x + R_0) \notag\\
  & = A (x - \rho_1^2)(x - \rho_2^2)(x - \rho_3^2)(x - \rho_4^2)(x - \rho_5^2)(x - \rho_6^2) \notag
\end{align}
In practice, we will choose roots $\rho_i^2$ and then determine the polynomial $D_T(x)$; and from it, the coefficients $a, b, c, A, B, C, D$. Note that in the above we are free to choose $B, C, D$, so matching coefficients for the $x^5, x^4$ and $x^3$ terms do not add any additional constraints. So we must simultaneously solve the following three equations in $\OK$:
\begin{align*}
  2ac + b^2 &= R_2 A, \\
  2bc       &= R_1A, \\
  c^2       &= R_0 A. 
\end{align*}
So long as this system of Diophantine equations is solvable in $\OK$, we may construct such an elliptic surface. In section \ref{sec:examples} below, we provide some examples of elliptic surfaces over number fields.


\section{Quadratic Legendre Sums}

The following proposition on quadratic Legendre sums in finite fields is the generalization of Lemma A.1 from \cite{AL-RM} to number fields (see also \cite{BEW}). Let $q=p^r$ be an odd prime power, and assume $\OK / \p = \F_q$ is a finite field with $q$ elements. Let $\left( \frac{\cdot}{q} \right)$ denote the $\F_q$-Legendre symbol which indicates whether or not an element of $\F_q$ is a square. 

\begin{prop}
  \label{importantlemma} 
  If $a \in \OK$ is not zero modulo $\p$, then
  \[
    \sum_{t \in \F_q} \left(\frac{at^2 + bt + c}{q}\right) = 
    \begin{cases}
      (q - 1) \left(\frac a q\right) & \text{if } b^2 - 4ac \equiv 0 \bmod p\\
      - \left(\frac a q \right) & \text{otherwise.}
    \end{cases}
  \]
\end{prop}

\begin{proof}
  The first case is straightforward, as if $b^2-4ac\equiv0 \bmod \p$, then $at^2+bt+c = a(t-t')^2$ for some $t' \in \F_q$, and each of the terms in the sum except $t'$ contribute $\left(\frac{a}{q}\right)$, and $t'$ contributes $0$. 
  
  For the other case, when $b^2 - 4ac \not \equiv 0 \bmod \p$, we first reinterpret the sum as counting points on the conic $C: s^2 = at^2+ bt + c$ in the following way:
\[
    \#C(\F_q)\ =\ \sum_{t \in \F_q} \left(1+ \left(\frac{at^2+bt+c}{q} \right)\right) \ =\ q + S. 
\]
  Here $S$ is the sum of interest. It is well-known that a nondegenerate conic of this particular form always has a rational point over $\F_q$ when $q$ is a power of an odd prime \cite{E}, \cite[Theorem 3.4]{Su}. From this, we may parameterize all rational points using some line that does not meet the original rational point. This gives at most $q+1$ points on the curve. However, this parametrization introduces a denominator that is possibly quadratic in $t$, which means at most $2$ rational points on the line might not correspond to rational points on the curve. Thus we have
  \[
    q-1\ \leq\ \#C(\F_q) \ \leq\ q+1,
  \]
  which gives
  \[ 
    -1 \ \leq\ S \ \leq\ 1. 
  \]

To determine the value of $S$, we compute it is modulo $p$. By Euler's criterion in finite fields:
\[
  S \ \equiv\ \sum_{t \in \F_q} (at^2 + bt+ c)^{\frac{q-1}{2}}\ \equiv\ \sum_{t \in \F_q} \left(\frac{a}{q}\right)t^{q-1} + r(t) \ \equiv\ -\legsym{a}{q} + \sum_i r_i \sum_{t \in \F_q} t^i,  
\]
where $r(t)$ is a polynomial of degree $<q-1$, and $r_i$ are the roots of $D_t(\p)$, as above. Each of the inner sums $\sum_{t \in \F_q} t^i$ is $0$, since $i<q-1$, so one of the terms is nonzero, but the sum is stable under multiplication by any of its summands. Thus $S \equiv -\legsym{a}{q} \pmod{p}$. Since $-1 \leq S \leq 1$, we have $S = -\legsym{a}{q}$, as desired.
\end{proof}

\section{Examples}
\label{sec:examples}

Let $K$ be an arbitrary number field. Theorem \ref{thm:main} claims that we may produce elliptic curves of rank $6$ over $K$; we provide a recipe to produce these curves over number fields in this section. 

 As in \cite[Section 2.2]{AL-RM}, we choose $A = 64 R_0^3$ for simplicity. This choice is convenient, because it allows us to solve 
\[\begin{array}{llll}
 c^2 &= 64 R_0^3 & \Rightarrow c &= 8 R_0^3 \\
 2bc &= 64 R_0^3 R_1 & \Rightarrow b &= 4 R_0 R_1\\
 2ac + b^2 &= 64 R_0^3 R_2 & \Rightarrow a & = 4 R_0 R_2 - R_1^2.
\end{array}\]
Additionally, we may solve for $B, C, D$ in terms of $R_0, \ldots, R_5$. Altogether, we have
\begin{align}
a & = 4 R_0 R_2 - R_1^2\notag\\
b & = 4 R_0 R_1\notag\\
c & = 8 R_0^2\notag\\
A & = 64 R_0^3 \label{eq:solvedcoeffs}\\
B & = A R_5 - 2a\notag\\
C & = A R_4 - a^2 - 2b\notag\\
D & = A R_3 - 2ab - 2c\notag
\end{align}
The above determines the coefficients of the elliptic curve defined by the equations \eqref{eq:curvedef} in terms of the roots $r_i = \rho_i^2$ of the discriminant $D_T(x)$, as in \eqref{eq:discriminantcoefficients}. 

Expanding the first line of \eqref{eq:curvedef}, we arrive at the following equation for the elliptic curve. 
\begin{equation}
\label{eq:expandedcurveequation}
y^2 = x^3 + (2aT - B)x^2 + (2bT - C)(T^2 + 2T - A + 1)x + (2cT - D)(T^2 + 2T - A + 1)^2
\end{equation}
To produce curves over $K$, we may use the following recipe:
\begin{itemize}
\setlength{\itemsep}{1em}
\item choose six squares $r_1 = \rho_1^2, \ldots, r_6 = \rho_6^2$ in $K$ to be the roots of $D_T(x)$; 
\item solve for $R_0, \ldots, R_5$ as the coefficients of the degree-six polynomial $D_T(x)$; 
\item use \eqref{eq:solvedcoeffs} to find $a, b, c, A, B, C, D$; 
\item plug these values into \eqref{eq:expandedcurveequation} to determine the equation for the elliptic curve. 
\item Specialize at $T = t$ for some $t \in K$. For generic choices of $t$, this specializes to an elliptic curve of rank at least $6$, by Silverman's specialization theorem \cite[Theorem 11.4]{Si2}. 
\end{itemize}
We carry out this procedure for a few choices of number fields below.  

\begin{example}
Let $K = \QQ$. In \cite[Theorem 2.1]{AL-RM}, a rank 6 elliptic surface $\E$ over $\QQ(T)$ with equation as in \eqref{eq:expandedcurveequation} and $r_\ell = \ell^2$, for $\ell = 1, 2, \ldots, 6$. 
\[\begin{array}{rlrl}
a &= 166601111104, & A &= 8916100448256000000,\\
b &= -1603174809600, \hspace*{1cm}& B &= -811365140824616222208,\\
c &= 2149908480000, & C &= 26497490347321493520384,\\
& & D &= -343107594345448813363200.
\end{array}\]

\end{example}

\begin{example}
Choose $K = \QQ(i)$, $i = \sqrt{-1}$. Then with the choices $\rho_1 = 1 + i$, $\rho_s = s$ for $s = 2, 3, \ldots, 6$; $r_\ell = \rho_\ell^2$, we find 
\begin{align*}
a & = - 353892105216+ 528220569600i \\
b & =  2112882278400-2149908480000i \\
c & = -8599633920000\\
A & = -71328803586048000000i\\
B & =  153634690402938169196544+ 380285771115321439027200i \\
C & =  166616532655598905196544 + 166085373946419295027200i \\
D & = - 1191348658308947587891200-789381960170093936640000i 
\end{align*}
Via \eqref{eq:expandedcurveequation}, this determines a rank 6 elliptic curve $\mathcal E$ over $K(T) = \QQ(i)(T)$. The $j$-invariant of this curve is $j(\mathcal E) = p(T)/q(T)$, where $p(T)$ and $q(T)$ are degree $9$ and $10$ polynomials in $T$, respectively. The leading coefficient of $p(T)$ is 
\begin{align*}
 p_9 &= 17575652563096624654081015917110624256000000 \\
 & \hspace{1cm} + 16682776400034638205353357277029990400000000i. 
\end{align*}
In particular, $j(\mathcal E) \in K(T)$, but $j(\E) \not \in \QQ(T)$. 
\end{example}

\begin{example}
Choose $K = \QQ(\zeta_5)$, where $\zeta_5$ is a fifth root of unity. Then with the choices $r_\ell = \rho_\ell^2$, 
\[ \rho_1 = \zeta_5, \ \ \rho_2 = \zeta_5^2, \ \ \rho_3 = \zeta_5^3, \ \ \rho_4 = 1 + \zeta_5, \ \ \rho_5 = 1 + \zeta_5^2, \ \ \rho_6 = 1 + \zeta_5^3. \]
the coefficients are 
\begin{align*}
a & = 55\zeta_5^3 - 8\zeta_5^2 + 41\zeta_5 + 25 \\
b & = -12\zeta_5^3 - 40\zeta_5^2 + 20\zeta_5 - 48\\
c & = -24\zeta_5^2 + 16\zeta_5 - 24\\
A & = -832\zeta_5^3 - 320\zeta_5^2 - 320\zeta_5 - 832\\
B & = 2312\zeta_5^3 + 3693\zeta_5^2 - 861\zeta_5 + 5117\\
C & = -2040\zeta_5^3 + 1837\zeta_5^2 - 2397\zeta_5 + 573\\
D & = 2440\zeta_5^3 + 152\zeta_5^2 + 1408\zeta_5 + 1664
\end{align*}
This determines a rank 6 elliptic curve $\mathcal E$ over $K(T) = \QQ(\zeta_5)(T)$ via \eqref{eq:expandedcurveequation}. The $j$-invariant for this elliptic curve is 
$j(\mathcal E) = p(T)/q(T)$
where $p(T)$ and $q(T)$ are degree 9 and 10 polynomials in $T$, respectively. The leading coefficient of $p(T)$ is 
\begin{align*}
p_9 &=-1203674209337006199645159424 \zeta_{5}^{3} + 470942041084292914570780672 \zeta_{5}^{2} \\
&\hspace{2cm} - 1034969760873268271839698944 \zeta_{5} - 272969531667632951696109568, 
\end{align*}
witnessing the fact that $j(\E) \in K(T)$, but $j(\mathcal E) \not \in \QQ(T)$. 
\end{example}


\end{document}